\documentclass[review]{elsarticle}

\usepackage{bm,amssymb,amsmath,mathrsfs}
\usepackage{amsthm}
\usepackage{lineno}
\usepackage[usenames]{color}

\usepackage{dcolumn}

\theoremstyle{plain}
\newtheorem{theorem}{Theorem}[section]
\newtheorem*{theorem*}{Theorem}
\newtheorem{proposition}[theorem]{Proposition}
\newtheorem{lemma}[theorem]{Lemma}
\newtheorem{corollary}[theorem]{Corollary}
\newtheorem{remark}[theorem]{Remark}

\newtheorem{example}[theorem]{Example}

\begin{document}

\journal{(internal report CC25-15)}

\begin{frontmatter}

\title{Type II success runs of Bernoulli trials separated by a gap}

\author[usc]{S.~J.~Dilworth}
\ead{dilworth@math.sc.edu}
\address[usc]{Department of Mathematics, University of South Carolina, Columbia, SC 29208, USA}

\author[cc]{S.~R.~Mane}
\ead{srmane001@gmail.com}
\address[cc]{Convergent Computing Inc., P.~O.~Box 561, Shoreham, NY 11786, USA}

\begin{abstract}
We treat success runs of independent identically distributed Bernoulli trials (with success parameter $p$)
distributed according to the Type II binomial distribution of order $k$.
However, the success runs are separated by a gap $g\ge1$ (a failure followed by $g-1$ arbitrary outcomes).
Most of the literature treats the case $g=1$ only.
Our main results are expressions for the probability mass function (we present two derivations)
and the distribution of the longest success run.
We also present more concise expressions for previously published results for the factorial moments.
We present results for the mean, variance, probability mass function and factorial moments for
$\textrm{NB}_{\rm II}(k,g,r)$, the Type II negative binomial distribution of order $k$,
where the number of success runs $r$ is fixed and the number of trials $n$ is variable.
Let $L$ denote the length of the longest success run.
We present a recurrence and generating function for the distribution of $L$
and derive expressions for the mean, variance and factorial moments of $L$.
\end{abstract}


\end{frontmatter}

\newpage
\setcounter{equation}{0}
\section{Introduction}\label{sec:intro}
The Type II binomial distribution of order $k$ is characterized by a fixed number $n$ of independent identically distributed Bernoulli trials (with success parameter $p$)
where a success run is a sequence of at least $k$ uninterrupted successes, and distinct success runs are separated by at least one failure.
(For this reason, this is also called the ``at least'' counting scheme, see e.g.~Antzoulakos and Chadjikonstantinidis \cite{AntzoulakosChadjikonstantinidis2001}.)
An important reference is the monograph on success runs by Balakrishnan and Koutras \cite{BalakrishnanKoutras}.
Significant results were published by de Moivre \cite{deMoivre} and Muselli \cite{Muselli}, and we shall make contact with their results below.
Hald \cite{Hald} and Malinovsky \cite{Malinovsky} present more recent and accessible references to de Moivre's work.

Han and Aki \cite{HanAki2000} pointed out that repairs to a system are not always performed instantaneously after a failure;
there can be a delay, and this can be modeled by a gap between success runs.
It is possible that regulations require a safety inspection, etc., and this can also be modeled by a gap between success runs.
Dilworth and Mane \cite{DM5} (hereafter `DM') studied the Type II binomial distribution of order $k$,
where the success runs were separated by a gap of $g\ge1$ trials, where the gap begins with a failure and is followed by $g-1$ arbitrary outcomes.
DM published expressions for the double pgf and the factorial moments, etc.
Here we present additional results on the subject.
Our main results are the probability mass function (two derivations, in Secs.~\ref{sec:pmf1} and \ref{sec:pmf2})
and the distribution of the longest success run (Secs.~\ref{sec:longestrun1}, \ref{sec:longestrun2} and \ref{sec:longestrun3}).

To clarify the concepts below, we quantify the notion of a gap more precisely.
The term `order $k$' means a success run has length at least $k$.
The gap applies only to strings of consecutive successes with length $t$, where $t \le n$. 
We can define the success runs (of order $k$ with gap $g$) inductively as follows.
The first sucess run  (if it exists) is the first  run of at least $k$ successes
which cannot be extened to a longer run.
We say that a success run which includes trial $m \le n$  terminates at trial $m$ if either $m=n$ or  trial $m+1$ is a failure.
Note that in this case  trial $m$ is the last success trial in the run.
Suppose  that $j \ge 1$ and that the $j^{th}$ success run has been defined inductively  and terminates at trial $m$.
Then the $(j+1)^{th}$ success run  (if it exists) is the first run of at least $k$ successes
which begins later than  trial $m+g$ and which cannot be extended to a longer run (beginning later than trial $m+g$).
Either there are no success runs or there exists a largest value of $r \ge 1$ such that there is a unique  $j^{th}$ success run for each $1 \le j \le r$.

\begin{proposition}\label{prop:max_numruns}
For fixed $n$, $k$ and $g$ the maximum number of success runs is $\lfloor (n+g)/(k+g) \rfloor$.
\end{proposition}
\begin{proof}
The maximum number of success runs is attained when every success run has length exactly $k$ (the shortest possible),
the gap length between consecutive success runs is exactly $g$ (the shortest possible).
Suppose there are $r$ success runs, then they are separated by $r-1$ gaps, hence we must have $kr+(r-1)g \le n$, i.e. ~$r \le \lfloor (n+g)/(k+g) \rfloor$
and the result follows.
\end{proof}
\begin{remark}
Antzoulakos and Chadjikonstantinidis (\cite{AntzoulakosChadjikonstantinidis2001}, Theorem 3.5)
and
Muselli (\cite{Muselli}, Theorem 1)
both treated the case $g=1$ and obtained the upper limit $\lfloor (n+1)/(k+1) \rfloor$.
We generalize their results.
\end{remark}

We shall derive a result below for $P(L \le t)$, the probability that the length $L$ of the longest success run does not exceed $t$.
The result depends on both $k$ and $t$, and one must make a distinction between them.
Such a distinction between strings of consecutive successes of different lengths does not apply if $g=1$, because all such strings are separated by a single failure.
This is the case treated in most of the literature, e.g.~by Muselli and de Moivre, etc.

The structure of this paper is as follows.
In Sec.~\ref{sec:pmf1}, we employ the double pgf (probability generating function)
to derive an expression for the probability mass function to attain $r$ success runs, for fixed $(k,g,n)$.
In Sec.~\ref{sec:facmom}, we revisit the results for the factorial moments derived in \cite{DM5}.
We employ the insights in this note to elucidate the structure of the factorial moments and rewrite them in a more concise and illuminating manner.
Let $L$ denote the length of the longest success run.
Sec.~\ref{sec:longestrun1} presents the derivation for $P(L \le t)$, for fixed $(k,g,n)$.
Sec.~\ref{sec:pmf2} presents an alternative derivation of the probability mass function, using the inclusion-exclusion pinciple.
Sec.~\ref{sec:negbinom} presents a few remarks on $\textrm{NB}_{\rm II}(k,g,r)$, the Type II negative binomial distribution of order $k$,
where the number of success runs $r$ is fixed and the number of trials $n$ is variable.
Sec.~\ref{sec:longestrun2} presents a recurrence and generating function for $P(L \ge t)$, 
and expressions for the mean, variance and factorial moments of $L$.
Sec.~\ref{sec:longestrun3} treats the case $g=1$ only and presents an alternative expression for the distribution of the longest success run,
leveraging the formalism developed in \cite{DM3}.
Sec.~\ref{sec:conc} concludes.

\setcounter{equation}{0}
\section{Probability mass function (pmf) I}\label{sec:pmf1}
We present the single and double pgf for the number of success runs.
We refer the reader to the text by Balakrishnan and Koutras \cite{BalakrishnanKoutras}
for details on the theory of probability generating functions for success runs.
For brevity, we omit mention of $k$ and $g$ in the indexing and 
let $p_{n,r}$ denote the probability of $r$ success runs in fixed $n$ trials.
DM denoted the single pgf via 
$\phi_n(t) = \sum_{j=0}^\infty p_{n,j}t^j$
and the double pgf as 
$\Phi(z,t) = \sum_{n=0}^\infty z^n\phi_n(t)$
(\cite{DM5}, eqs.~(2.1) and (2.2), respectively).
DM derived the following expression for $\Phi(z,t)$ (\cite{DM5}, eq.~(2.6))
\begin{equation}
\label{eq:doublepgf}
\Phi(t,z) = \frac{(1-z)(1+p^k(t-1)z^{k})+p^kqtz^{k+1}(1-z^{g-1})}{(1-z)(1-z+p^kqz^{k+1}(1-tz^{g-1}))} \,.
\end{equation} 
For $g=1$, eq.~\eqref{eq:doublepgf} simplifies to eq.~(5.11) of Balakrishnan and Koutras \cite{BalakrishnanKoutras}.
DM processed eq.~\eqref{eq:doublepgf} to obtain the factorial moments.
The generator of the factorial moments is obtained by evaluating $\partial^r\Phi(z,t)/\partial t^r$ at $t=1$.
The generator of the pmf is obtained by evaluating $\partial^r\Phi(z,t)/\partial t^r$ at $t=0$.
As noted in \cite{DM5}, $\Phi(t,z)$ is a rational function of linear polynomials in $t$.
Then, with obvious expressions for $A$, $B$, $C$ and $D$, 
\begin{equation}
\begin{split}
\Phi(t,z) = \frac{A+Bt}{C+Dt} 
&= \frac{AD-BC}{D}\,\frac{1}{C+Dt} +\frac{B}{D} 
\\
&= \frac{1 -pz}{qz^g(1-z)}\frac{1-z +qp^kz^{k+1} -qp^kz^{k+g}}{1-z +qp^kz^{k+1} -qp^kz^{k+g}t} 
-\frac{1 -pz -qz^g}{qz^g(1-z)} \,.
\end{split}
\end{equation}
This can easily be differentiated partially $r$ times with respect to $t$.
\begin{equation}
\label{eq:doublepgf_abcd_diff_t}
\frac{\partial^r \Phi(t,z)}{\partial t^r}
= r!\,\frac{1 -pz}{qz^g(1-z)}\frac{(1-z +qp^kz^{k+1} -qp^kz^{k+g}) q^rp^{rk}z^{r(k+g)}}{(1-z +qp^kz^{k+1} -qp^kz^{k+g}t)^{r+1}} \,.
\end{equation}
Divide by $r!$ and set $t=0$ to obtain the generator of the pmf, say $P^{(r)} = \sum_{n=0}^\infty z^n p_{n,r}$.
We omit the details of the algebra.
The solution for $p_{n,r}$ is
\begin{equation}
\label{eq:p_nr_Muselli_ish_alt}  
\begin{split}
  p_{n,r} &= \sum_{m=r}^{\left\lfloor \frac{n+1 -(r-1)(g-1)}{k+1} \right\rfloor} (-1)^{m-r} q^{m-1}p^{mk} \binom{m-1}{r-1} \quad\times
  \\
  &\qquad\qquad\qquad\qquad
  \biggl[\binom{n-mk+1 -(r-1)(g-1)}{m} -p\binom{n-mk -(r-1)(g-1)}{m}\biggr]
  \\
  &\quad
  -\sum_{m=r+1}^{\left\lfloor \frac{n+1 -r(g-1)}{k+1} \right\rfloor} (-1)^{m-r-1} q^{m-1}p^{mk} \binom{m-1}{r} \quad\times
  \\
  &\qquad\qquad\qquad\qquad
  \biggl[ \binom{n-mk+1 -r(g-1)}{m} -p\binom{n-mk -r(g-1)}{m}\biggr] \,.
\end{split}
\end{equation}
\begin{corollary}
  Define $\nu = n-(r-1)(g-1)$ for brevity. 
  Let $N_s$ denote the number of success runs, then with an obvious notation
\begin{equation}
\label{eq:p_N_gt_r}
\begin{split}
  P(N_s \ge r) &= \sum_{m=r}^{\left\lfloor \frac{\nu+1}{k+1} \right\rfloor} (-1)^{m-r} q^{m-1}p^{mk} \binom{m-1}{r-1} 
  \biggl[\binom{\nu-mk+1}{m} -p\binom{\nu-mk}{m}\biggr] \,.
\end{split}
\end{equation}
Observe that $P(N_s \ge r)$ depends on $n$ and $g$ only in the combination $n-(r-1)(g-1)$.
\end{corollary}
\begin{proof}
Clearly, $p_{n,r} = P(N_s \ge r) - P(N_s \ge r+1)$.
Next, observe that the second sum in eq.~\eqref{eq:p_nr_Muselli_ish_alt}
is the same as the first, with the transcription $r\gets r+1$.  
This reveals that the first sum in eq.~\eqref{eq:p_nr_Muselli_ish_alt} is the \emph{probability that the number of success runs is $r$ or more},
and eq.~\eqref{eq:p_N_gt_r} follows immediately.
\end{proof}
For $g=1$, the expression in eq.~\eqref{eq:p_nr_Muselli_ish_alt} simplifies and is
\begin{equation}
\label{eq:p_nr_Muselli_ish_g1_alt}  
\begin{split}
  p_{n,r}(g=1) &= \sum_{m=r}^{\left\lfloor \frac{n+1}{k+1} \right\rfloor} (-1)^{m-r} q^{m-1}p^{mk} \binom{m}{r} \biggl[\binom{n-mk+1}{m} -p\binom{n-mk}{m}\biggr] \,.
\end{split}
\end{equation}
Employing the identity $\binom{n}{m} = \binom{n-1}{m}+\binom{n-1}{m-1}$ yields
\begin{equation}
\label{eq:p_nr_Muselli}  
\begin{split}
  p_{n,r}(g=1) &= \sum_{m=r}^{\left\lfloor \frac{n+1}{k+1} \right\rfloor} (-1)^{m-r} q^{m-1}p^{mk} \binom{m}{r} \biggl[\binom{n-mk}{m-1} +(1-p)\binom{n-mk}{m}\biggr]
  \\
  &= \sum_{m=r}^{\left\lfloor \frac{n+1}{k+1} \right\rfloor} (-1)^{m-r} q^{m-1}p^{mk} \binom{m}{r} \biggl[\binom{n-mk}{m-1} +q\binom{n-mk}{m}\biggr] \,.
\end{split}
\end{equation}
This is Muselli's result (\cite{Muselli}, Theorem 3).

\begin{corollary}
  Set $r = 1$, then $P(N_s \ge 1)$ is the probability that there is at least one success run with length at least $k$.
  This is equivalently the probability that the longest success run has length at least $k$.
  Let $L$ denote the length of the longest success run. Then, recalling $\nu=n$ if $r=1$,
\begin{equation}
\label{eq:P_L_ge_k}
\begin{split}
  P(L \ge k) = P(N_s \ge 1) &= \sum_{m=1}^{\left\lfloor \frac{n+1}{k+1} \right\rfloor} (-1)^{m-1} q^{m-1}p^{mk} \biggl[\binom{n-mk+1}{m} -p\binom{n-mk}{m}\biggr] \,.
\end{split}
\end{equation}
Observe that $P(L \ge k)$ does not depend on the value of $g$. 
\end{corollary}
de Moivre \cite{deMoivre} calculated $P(L \ge k)$ for the case $g=1$.
A more accessible account of de Moivre's work is given by Hald \cite{Hald},
who stated that de Moivre gave his solution using a generating function without indicating his proof,
but de Moivre gave several numerical examples, obtaining the correct answer.
Malinovsky \cite{Malinovsky} has also given an accessible account of de Moivre's solution.
In particular, if $k \ge n/2$, then only one success run of length $k$ or more is possible and we have the simple solution
(\cite{Malinovsky}, Corollary 1)
\begin{equation}
\label{eq:Mal_cor1}
  P(L \ge k) = p^k(1 + (n-k)q) \qquad (\lfloor (n/2)\rfloor \le k \le n) \,.
\end{equation}
\begin{remark}
Malinovsky derived eq.~\eqref{eq:Mal_cor1} for the case $g=1$.
However, if the value of $k$ is so large that the trials contain only one success run of length $k$ or more,
then the value of the gap $g$ is irrelevant.
Hence eq.~\eqref{eq:Mal_cor1} is valid for all gap values $g\ge1$.
In this context, $k\ge \lfloor (n/2)\rfloor$ is the tightest lower bound on the value of $k$, independent of the value of $g$.
When $g >1$, it is tempting to argue that eq.~\eqref{eq:Mal_cor1} is valid provided $2k+g > n$, i.e.~$k\ge\lfloor(n-g)/2\rfloor$,
since for such $n$ it is not possible to have two runs of $k$ successes separated by a gap of size $g$.
But consider $n = 2k+1<2k+g$.
The sequence consisting of two runs of $k$ successes separated by a single $F$ will be counted twice in
eq.~\eqref{eq:Mal_cor1} and hence eq.~\eqref{eq:Mal_cor1} overestimates $P(L  \ge k)$ if $k<\lfloor(n/2)\rfloor$.
\end{remark}
Next, Muselli calculated the probability $P(L \le k-1)$, i.e.~the longest success run has fewer than $k$ trials.
His result is  (\cite{Muselli}, Corollary 1)
\begin{equation}
\label{eq:Mus_Cor1}
\begin{split}
  P(L \le k-1) = \sum_{m=0}^{\left\lfloor \frac{n+1}{k+1} \right\rfloor} (-1)^m q^{m-1}p^{mk} \biggl[\binom{n-mk}{m-1} +q\binom{n-mk}{m}\biggr] \,.
\end{split}
\end{equation}
\begin{remark}
Clearly $P(L \ge k)$ and $P(L \le k-1)$ sum to unity.
As a self-consistency check, let us verify that the expressions for $P(L \ge k)$ in eq.~\eqref{eq:P_L_ge_k} and for $P(L \le k-1)$ in eq.~\eqref{eq:Mus_Cor1} sum to unity.
First we process eq.~\eqref{eq:Mus_Cor1} using the binomial identity $\binom{n}{m-1} = \binom{n+1}{m} - \binom{n}{m}$ to obtain
\begin{equation}
\label{eq:Mus_Cor1_modify}
\begin{split}
  P(L \le k-1) &= \sum_{m=0}^{\left\lfloor \frac{n+1}{k+1} \right\rfloor} (-1)^m q^{m-1}p^{mk} \biggl[\binom{n-mk+1}{m} +(q-1)\binom{n-mk}{m}\biggr] 
  \\
  &= \sum_{m=0}^{\left\lfloor \frac{n+1}{k+1} \right\rfloor} (-1)^m q^{m-1}p^{mk} \biggl[\binom{n-mk+1}{m} -p\binom{n-mk}{m}\biggr] \,.
\end{split}
\end{equation}
Next we add the expressions in eqs.~\eqref{eq:Mus_Cor1_modify} and \eqref{eq:P_L_ge_k}:
\begin{equation}
\begin{split}
  P(L \le k-1) +P(L \ge k) &= \sum_{m=0}^{\left\lfloor \frac{n+1}{k+1} \right\rfloor} (-1)^m q^{m-1}p^{mk} \biggl[\binom{n-mk+1}{m} -p\binom{n-mk}{m}\biggr]
  \\
  &\qquad + \sum_{m=1}^{\left\lfloor \frac{n+1}{k+1} \right\rfloor} (-1)^{m-1} q^{m-1}p^{mk} \biggl[ \binom{n-mk+1}{m} -p \binom{n-mk}{m} \biggr] 
  \\
  &= q^{-1} (1-p) + \sum_{m=1}^{\left\lfloor \frac{n+1}{k+1} \right\rfloor} (-1)^m q^{m-1}p^{mk} \biggl[\binom{n-mk+1}{m} -p\binom{n-mk}{m}\biggr]
  \\
  &\qquad - \sum_{m=1}^{\left\lfloor \frac{n+1}{k+1} \right\rfloor} (-1)^m q^{m-1}p^{mk} \biggl[ \binom{n-mk+1}{m} -p \binom{n-mk}{m} \biggr] 
  \\
  &= 1 \,.
\end{split}
\end{equation}
\end{remark}
\begin{remark}
To close this section, the authors made a comment in \cite{DM5} that 
an expression for the pgf published by Antzoulakos and Chadjikonstantinidis \cite{AntzoulakosChadjikonstantinidis2001} contained a misprint.
Upon closer examination, the expression in the unnumbered equation before eq.~(3.4) in \cite{DM5} is erroneous
and the formula in \cite{AntzoulakosChadjikonstantinidis2001} is correct.
The error is regretted.
\end{remark}

\setcounter{equation}{0}
\section{Factorial moments}\label{sec:facmom}
Observe that the expression for $p_{n,r}$ in eq.~\eqref{eq:p_nr_Muselli_ish_alt} has the structure
$p_{n,r} = f(n) - pf(n-1)$, for a suitably defined function $f(n)$.
This can be traced to the term $1-pz$ in eq.~\eqref{eq:doublepgf_abcd_diff_t}.
Expanding in powers of $z$, it follows that the coefficient of $z^n$ has the structure 
$f(n) - pf(n-1)$, for a suitably defined function $f(n)$.
Clearly, the same structure is also true for $P(N_s \ge r)$ in eq.~\eqref{eq:p_N_gt_r}.

The factorial moments are also derived from eq.~\eqref{eq:doublepgf_abcd_diff_t}, hence they also have the same structure.
One can see this in the expression for the factorial momemt $F_n^{(r)}$ in (\cite{DM5}, eq.~(3.1), but it is not self-evident.
We obtain a more elegant expression by employing $\nu=n-(r-1)(g-1)$. Then
\begin{equation}
\label{eq:DM5facmom_nu}
\begin{split}
F_n^{(r)} &= r! \biggl\{ \sum_{m=r}^{\lfloor(\nu+1)/(k+1)\rfloor}
(-1)^{m-r} q^{m-1} p^{mk} \binom{m-1}{r-1}
\sum_{s=0}^{m-r} (-1)^s \binom{m-r}{s} \binom{\nu +1 -mk -s(g-1)}{m} 
\\
&\qquad\quad
-p\,\sum_{m=r}^{\lfloor\nu/(k+1)\rfloor}
(-1)^{m-r} q^{m-1} p^{mk} \binom{m-1}{r-1}
\sum_{s=0}^{m-r} (-1)^s \binom{m-r}{s} \binom{\nu -mk -s(g-1)}{m} \biggr\} \,.
\end{split}
\end{equation}
The second line is the same as the first with $\nu$ in place of $\nu+1$ (and multiplied by $p$).
Observe that $n$ always appears in the combination $n-(r-1)(g-1)$, but $g$ also appears independently in the binomial coefficients.

We obtain a more concise expression if we raise the upper limit in the second sum to equal that in the first.
This introduces extraneous terms in the second sum, but they are all zero because the numerator is less than the denominator in the extraneous binomial coefficients.
(We lose the $f(n)-pf(n-1)$ structure, but that is not important.)
Then we obtain
\begin{equation}
\label{eq:DM5facmom_concise}
\begin{split}
F_n^{(r)} &= r! \sum_{m=r}^{\left\lfloor \frac{\nu+1}{k+1} \right\rfloor}
(-1)^{m-r} q^{m-1} p^{mk} \binom{m-1}{r-1} \;\times
\\
&\qquad\qquad
\sum_{s=0}^{m-r} (-1)^s \binom{m-r}{s} \biggl[ \binom{\nu+1-mk -s(g-1)}{m} -p\binom{\nu-mk -s(g-1)}{m} \biggr] \,.
\end{split}
\end{equation}
This has a very similar structure to eq.~\eqref{eq:p_N_gt_r}.
For $g=1$, the above expression simplifies greatly and is
\begin{equation}
F_n^{(r)}(g=1) = r! q^{r-1} p^{kr} \biggl[ \binom{n-kr+1}{r} -p\binom{n-kr}{r} \biggr] \,.
\end{equation}
This agrees with Antzoulakos and Chadjikonstantinidis (\cite{AntzoulakosChadjikonstantinidis2001}, Prop.~4.3),
as was noted already in \cite{DM5}.

\setcounter{equation}{0}
\section{Distribution of longest success run I}\label{sec:longestrun1}
  Let $L$ be the length of the longest run, for fixed $(k, g, n)$. We compute $P(L \le t)$.
  Recall that Muselli \cite{Muselli} and de Moivre \cite{deMoivre} published such results for $g = 1$,
  but as explained in the introduction, for $g > 1$ we must make a distinction between $k$ and $t$.
  For completeness, we start with a short derivation of Muselli's Theorem 2 using the inclusion-exclusion formula, expressed in our notation.
Muselli's instructive proof was less direct but more ingenious. He also obtained an even  simpler expression for $P(L \le t)$ as Corollary~1 later in his article.
\begin{theorem}\cite[Theorem 2]{Muselli}\label{thm: Mus} For $g=1$ and $t\le n$
\begin{equation} \label{eq: Mus}
P(L \le t) = \sum_{u=0}^{n} p^{n-u} q^{u}\sum_{s=0}^{\min(u+1,\lfloor (n-u)/(t+1) \rfloor)} (-1)^s  \binom{u+1}{s}\binom{n-s(t+1)}{u}.
\end{equation}
\end{theorem}
\begin{proof} Let $u$ be the exact number of $F$ trials among the first $n$ trials, so $0 \le u \le n$. The $F$ trials create $u+1$ `slots'  between them for potential success runs, including a slot before the first  and after the last $F$ trial. Consider a specific subcollection of $s$ slots, where $0 \le s \le u+1$. Suppose each slot from this subcollection contains at least $t+1$ $S$ trials.  The number of  ways in which this can happen can be counted as follows: allocate $t+1$ $S$ trials to each slot from the given  subcollection and then allocate the remaining $n - u - s(t+1)$ $S$ trials arbitrarily to all $u+1$ slots.
  So there are
$\binom{n-s(t+1)}{u}$
ways. The given subcollection can be chosen in
$ \binom{u+1}{s}$ ways.
Hence by the usual inclusion-exclusion principle for the probability of the union of events  (\cite{Feller}, p.~61)
$$ P(L \le t) = \sum_{u=0}^{n} p^{n-u} q^{u}\sum_{s=0}^{\min(u+1,\lfloor (n-u)/(t+1) \rfloor)} (-1)^s  \binom{u+1}{s}\binom{n-s(t+1)}{u}.$$
Note that the upper limit on the sum over $s$   ensures that all the binomial coefficients are nonzero.
\end{proof}

The proof above does not generalize in a straightforward manner to the case $g>1$. Indeed, consider the following seemingly paradoxical example which illustrates the complications which arise for $g>1$.
\begin{example} For $k=2$ and $g=3$, consider the following possible  realization of $n=11$ trials:
$$ A = SSFSSSFSFSS, B=SFFSSSFSFSS.$$
Note that the only difference between $A$ and $B$ is the second trial which is $S$ in $A$ and $F$ in $B$. Let us consider the success runs in $A$ and $B$ respectively. Note that $r=2$ in both cases. For $A$ the success runs are $(1,2)$ and $(10,11)$, while $ (3,4,5)$ is the buffer after the first success run.
Hence $L=2$ for $A$. For $B$, on the other hand, the success runs are $(4,5,6)$  and  $(10,11)$, while $(7,8,9)$ is the buffer after the first success run.
Hence $L=3$ for $B$, which is seemingly paradoxical  because the $S$ trials for $B$ are a subset of those for $A$! Clearly, this phenomenon cannot occur when $g=1$. 

Let $C$ be obtained from $A$ by reversing the order of the trials, i.e.,
$$C = SSFSFSSSFSS.$$
Note that the success runs for $C$ are $(1,2)$ and $(6,7,8)$ and hence $L=3$.
In particular, the success runs for $C$ are not the same as the reversed success runs for $A$!
This is another phenomenon which obviously cannot occur for $g=1$.
\end{example}

Referring to the previous example, recall that  $(4,5,6)$ is a success run for $B$ but not for $A$. 
Let us say that   $(4,5,6)$ and $(10,11)$ are  `$g$-separated success blocks' for $A$.
Formally, a success block consists of  a sequence of at least  $k$ consecutive $S$ trials which 
either contains  the last $k$ trials   $(n-k+1, n-k+2,\dots,n)$ or is followed immediately by an $F$ trial. Clearly,  two distinct success blocks are either disjoint or one is an extension of the other.
  Consecutive  success blocks are $g$-separated if they are separated by a buffer of size at least $g$.  Clearly, the success runs constitute $r$ $g$-separated success blocks but success blocks are not necessarily success runs. Let us show that it is not possible to have more than $r$ $g$-separated success blocks.

\begin{proposition}
 Fix $k,g$ and $n$. Suppose that there are $r$ success runs  for a particular sequence of $n$ trials. Then there are at most $r$ $g$-separated 
success blocks. \end{proposition}
\begin{proof} The result is obvious if $r=0$. So suppose that $r \ge 1$ and that the result is true for all realizations  of $n \ge 1$ trials with $r-1$ success runs. Consider a collection of $s\ge 1$ $g$-separated success blocks. Suppose that the first success block terminates at trial
$m \le n$. Note that the first success run must  terminate at trial $m_1$, where $m_1 \le m$, and that the number of sucess runs after trial $m_1 +g$ equals $r-1$.
So by  inductive hypothesis the maximum number of $g$-separated success blocks starting after trial $m_1  +g$ is also $r-1$. Since $m+g \ge m_1 +g$, it follows that $s-1 \le r-1$, i.e., $s \le r$.
\end{proof}
\begin{corollary} Fix $k,g$ and $n$. Suppose there are $r$ success runs for a particular sequence of $n$ trials. Then there are also $r$ success runs for the reversed sequence.
\end{corollary}

After these preliminaries, let us  now consider the case $g>1$.
 Let $r \ge 0$ be the number of success runs of length at least $k$ separated by buffers of length $g$. Note that $r=0$ if and only if  $L=0$ and that either $L=0$ or $L \ge k$.

Suppose that $r \ge 1$ and that  there are a total of $rk+v$ $S$ trials in the $r$ success runs, where $v \ge 0$.  

We are  going to apply the inclusion-exclusion formula several times in this proof. To aid the reader's comprehension we will give all the details the first time but suppress them in subsequent applications of the inclusion-exclusion formula.

Note that for each given subcollection of $s$ success runs, where $0 \le s \le r$,  there are 
$$\binom{v - s(t-k+1) + (r-1)}{r-1}$$
ways of distributing the  `additional'  $v$ $S$ trials among all of the $r$ success runs so that  each success run of the given subcollection has more than $t$
successes.
A subcollection of $s$ success runs can be chosen in 
$ \binom{r}{s}$ ways.
It will be helpful below to employ $A(r, i, v)$, defined as follows:
\begin{equation}
\label{eq:A_def}
A(r,i,v) = \sum_{s=0}^{\min(r,v/i)} (-1)^s \binom{r}{s}\binom{v - is + (r-1)}{r-1}\,.
\end{equation}
Hence, by the inclusion-exclusion formula,  the number of ways of distributing  the `additional'  $v$ $S$ trials among the $r$ success runs so that $L \le t$ is given by 
$A(r,t-k+1,v)$.

Recall also $\nu = n -(r-1)(g-1)$. It will be helpful below to define an additional parameter $\nu^\prime = n -r(g-1)$
(essentially, replace $r \gets r+1$ in $\nu$).

Note that each buffer starts with an $F$ trial. Let $u$ be the total number of $F$ trials that either start a buffer or do not occur within a buffer.
Note that $u \ge r-1$.  There are $u+1$ `slots'  between them (including the slot before the first and after the last). Note that each slot contains either  no success runs or exactly  one success run. We consider three exhaustive and  mutually exclusive cases.

\textbf{Case I.} Each success run is followed by a completed buffer.

In this case $u \ge r$ and  the final slot does not contain a success run. Hence the  slots containing a  success run  can be chosen in $\binom{u}{r}$ ways. The remaining $u+1-r$ slots do not contain a success run.  Note that each buffer starts with an $F$ trial and is followed by $g-1$ arbitrary $S$ or $F$ trials. Hence there are  $n - u- kr - v - r(g-1)$ $S$ trials to be distributed among the remaining $u+1 -r$ slots so that each slot contains at most $k-1$ of the $S$ trials.  By the inclusion-exclusion formula, this can be done in 
$$ \sum_{s \ge 0} (-1)^s \binom{u+1-r}{s} \binom{n-(k+1)r-v-r(g-1)-sk}{u-r}$$ ways. Each of these realizations contains the $u$ $F$ trials, the  $r(g-1)$ arbitrary buffer trials, and the remaining  $n - u - r(g-1)$ $S$ trials. Hence the probability that $L \le t$ in Case I is given by  

\begin{equation} \begin{split}
B(r,t,n) &= \sum_{u \ge r} q^u p^{n-u-r(g-1)} \binom{u}{r} \sum_{v \ge 0} A(r,t-k+1,v)\biggl[ \sum_{s \ge 0} (-1)^s \binom{u+1-r}{s} \binom{n-(k+1)r-v-r(g-1)-sk}{u-r}\biggr]
\\
&= \sum_{u \ge r} q^u p^{\nu^\prime-u} \binom{u}{r} \sum_{v \ge 0} A(r,t-k+1,v) A(u-r+1, k, \nu^\prime -kr-u-v) \,.
\end{split} \end{equation}
Here $1 \le r \le n/(k+g)$,  $r \le u \le \nu^\prime -kr$, $0 \le v \le \nu^\prime -kr-u$.

\textbf{Case II.} The final success run occurs in the final slot (and hence is not followed  by an $F$ trial). 

In this case $u \ge r-1$. If $u = r-1$, then each slot is occupied by a success run and $v = n-kr - (r-1)g$. There are 
$ A(r,t-k+1,n-kr-(r-1)g)$ possibilities.

If $u \ge r$,  then  the slots containing the first $r-1$ success runs  can be chosen in  $\binom{u}{r-1}$ ways.  The remaining $u+1-r$ slots do not contain a success run. But now there are $n - kr - v -u  - (r-1)(g-1)$ $S$ trials to be distributed among them so that  each slot contains at most $k-1$ of the $S$ trials.  By the inclusion-exclusion formula, this  can be done in 
$$\sum_{s \ge 0} (-1)^s \binom{u+1-r}{s} \binom{n-k(r+s)-r-v-(r-1)(g-1)}{u-r}$$
ways. Each of these realizations contains the $u$ $F$ trials, the  $(r-1)(g-1)$ arbitrary buffer trials, and the remaining  $n - u - r(g-1)$ $S$ trials. Hence the probability that $L \le t$ in Case II is given by 

\begin{equation} \begin{split}
C(r,t,n) &= A(r,t-k+1,n-kr-(r-1)g)q^{r-1}p^{n-(r-1)g}
\\
&\quad +\sum_{u \ge r} q^u p^{n-u-(r-1)(g-1)} \binom{u}{r-1} \sum_{v\ge0}A(r,t-k+1,v) \;\times
\\
&\qquad\qquad\qquad
\biggl[ \sum_{s \ge 0} (-1)^s \binom{u+1-r}{s} \binom{n-k(r+s)-r-v-(r-1)(g-1)}{u-r}\biggr]
\\
&= A(r,t-k+1,n-kr-(r-1)g)q^{r-1}p^{n-(r-1)g}
\\
&\quad +\sum_{u \ge r} q^u p^{\nu-u} \binom{u}{r-1} \sum_{v\ge0}A(r,t-k+1,v) A(u-r+1,k,\nu -kr-u-v) \,.
\end{split} \end{equation}
Here $1 \le r \le (n+g)/(k+g)$,  $r \le u \le \nu -kr$, $0 \le v \le \nu -kr-u$.

\textbf{Case III.}
The final success run occurs in the penultimate slot and   its buffer is not completed.

In this case the final, i.e., $r^{th}$, success run is followed by an $F$ trial and then by $w$ arbitrary trials where $0 \le w  \le g-2$.
Note that the final success run is `in progress' at trial $n-1-w$ and then terminated at trial $n-w$ by an $F$ trial which occurs with probability $q$. Hence we may regard the first $n-1-w$ trials as an instance of Case II with $n$ replaced by  $n-w-1$. It follows that the required probability in Case III, for a given value of $r$, is given by 
$$  q \sum_{w=0}^{g-2}  C(r,t,n-w-1) = q \sum_{w=1}^{g-1} C(r,t,n-w).$$

  It remains only to calculate $P(r=0)$, the probability that there are no success runs of length $k$ or more.
  Clearly $P(r=0)$ is independent of $g$.
In fact $P(r=0)$ can be read off immediately  from eq.~\eqref{eq: Mus}  by setting $t=k-1$.
Hence
\begin{equation}  \label{eq: r=0}
  P(r=0)=  \sum_{u=0}^{n} p^{n-u} q^{u}\sum_{s=0}^{\min(u+1, \lfloor (n-u)/k\rfloor)} (-1)^s  \binom{u+1}{s}\binom{n-sk}{u} \,.
\end{equation}

\begin{remark}
Observe also that $P(L \le t) = P(r=0)$ if $t < k$ and Theorem \ref{thm: Mus} is really intended for $t \ge k$.
\end{remark}
\begin{remark} Note that $f(n)= P(r=0)$ satisfies  the recurrence relation
$$ f(n)-f(n-1)+ p^kq f(n-k-1) =0 \qquad (n \ge k+1) \,.$$
Hence eq.~\eqref{eq: r=0} is the solution to this recurrence 
with the  initial conditions  $f(j)=1$, $0 \le j \le k-1$, $f(k) =1- p^k$.
\end{remark}
Putting things together, we have proved the following theorem.
\begin{theorem}
The distribution of the longest success run is
\begin{equation}
  P(L \le t) = P(r=0) + \sum_{r=1}^{\lfloor(n+g)/(k+g)\rfloor}   \biggl[B(r,t,n) + C(r,t,n) + q \sum_{w=1}^{g-1} C(r,t,n-w)\biggr] \,,
\end{equation}
where $B$ and $C$ are as defined above.
\end{theorem}

\begin{remark}
The analog of eq.~\eqref{eq:Mal_cor1} is as follows.
We require both $t>k$ and $t > n-k$ (hence $n > 2k$), so that there is only one success run and it has length $t$ or more.
Then
\begin{equation}
\label{eq:P_L_ge_t_onerun}
  P(L \ge t) = p^t (1 + (n-t)q) \qquad (\max(k,n-k) < t \le n) \,.
\end{equation}
\end{remark}

\setcounter{equation}{0}
\section{Probability mass function II}\label{sec:pmf2}
In this section we provide  alternative formulas for  the probability mass and cumulative distribution  functions.  The proofs  do not use the double generating function derived in DM. At the end of the section we reconcile these results with those of Section~\ref{sec:pmf1}.

Let 
$M(k,g,n,r)$ denote the probability that a sequence of $n$ independent Bernoulli contains exactly $r$ runs of length at least $k$, each terminated by an $F$ trial,  and each  separated by at least $g$ trials from the next success run.
Muselli considered the  case $g=1$ and proved the following (\cite{Muselli}, Theorem 1):
\begin{equation} \label{eq: musellisformula}
M(k,1,n,r)=\sum_{i=r}^{\lfloor (n+1)/(k+1)\rfloor}(-1)^{i-r}\binom{i}{r}\sum_{j=i-1}^{n-ik}\binom{j+1}{i}\binom{n-ik}{j}p^{n-j}q^j. \end{equation}
 The goal of this section is to solve the case $g>1$ by reducing to $g=1$.  However, the reduction is not completely straightforward.

 To that end, 
let $X(k,g,n,r)$ denote the probability that the $r^{th}$ success run terminates at trial $n-1$, so that trial $n$ is an $F$ trial.
\begin{remark}\label{rem: prog}  Note that $\dfrac{1}{q}X(k,g,n+1,r)$ is equal to the probability that  the $r^{th}$ success run contains trial $n$, i.e. has not been completed by  time $n$.
\end{remark}

 The following lemma  straightforwardly reduces the calculation of $X(k,g,n,r)$ to the case $g=1$.
\begin{lemma}  \label{lem: reduction} $$X(k,g,n,r) = X(k,1,n-(r-1)(g-1),r).$$
\end{lemma}  \begin{proof} Suppose  that the $r^{th}$ sucess run terminates at trial $n-1$, so that trial $n$ is an $F$ trial. Replacing each of  the first
$r-1$ buffers by a single $F$ trial  yields a sequence of $n - (r-1)(g-1)$ trials containing $r$ success runs, with $g=1$, such that the final trial is an $F$ and the $r^{th}$ success run terminates at the penultimate trial. Conversely, every such sequence can be extended to a sequence of $n$ trials by replacing each $F$ trial following the first $r-1$ success runs by an arbitrary buffer of length $g$. The result follows from this correspondence.
\end{proof}

 Let us obtain an expression for $X(k,1,n,r)$   in the spirit of Muselli's formula for $M(k,1,n,r)$.
\begin{theorem}
$$X(k,1,n,r) = \sum_{u=r-1}^{n-1-kr}q^{u+1}p^{n-u-1} \biggl[\sum_{s=r-1}^u  (-1)^{s-r+1} \binom{u}{s} \binom {s}{r-1}\binom{n-1-(s+1)k}{u}\biggr].$$
\end{theorem}
 \begin{proof} The $n^{th}$ trial is an $F$.
Let  $u$ denote the total number of $F$
trials among the first $n-1$ trials. Note that $u \ge r-1$ and $u \le n-1-kr$. These  $F$ trials  generate $u+1$ slots.
The $r^{th}$ success run  occupies the final  slot.
Index the first $u$ slots by $1,2,\dots,u$. Suppose that $1 \le v_1 < v_2< \dots v_{r-1} \le u$ and that $\{v_1,\dots,v_{r-1}\} \subseteq A \subseteq \{1,2,\dots, u\}$. Let $|A|=s$. The number of ways of distributing the $n-u-1$ $S$ trials such that the final slot and each slot indexed by an element of $A$  contain a run is given by $$\binom{n-1-(s+1)k}{u}.$$ By a version of the inclusion-exclusion formula (cf.~\cite{Feller}, p.~64), the number of ways of distributing the $S$ trials such that there are \textit{exactly} $r$ success runs which occupy the final slot and the slots indexed by $\{v_1,\dots,v_{r-1}\}$ is given by
$$\sum_{A \supseteq\{v_1,\dots,v_{r-1}\}}(-1)^{s-r+1}\binom{n-1-(s+1)k}{u}.$$ Hence the number of ways of distributing the $S$ trials such that there are exactly $r$ success runs (including a run in the  the final slot) is given by
$$\sum_{\{v_1,\dots v_{r-1}\}}\sum_{A \supseteq \{v_1,\dots,v_{r-1}\}}(-1)^{s-r+1}\binom{n-1-(s+1)k}{u},$$ i.e., reversing the order of summation,
$$\sum_{s=r-1}^u  (-1)^{s-r+1} \binom{u}{s} \binom {s}{r-1}\binom{n-1-(s+1)k}{u}.$$
Here we use the fact that there are $\binom{u}{s}$ ways of choosing $A$ and $\binom{s}{r-1}$ of choosing $\{v_1,\dots,v_{r-1}\} \subseteq A$.
Each of these ways has probability $q^{u+1}p^{n-u-1}$. Summing over $u$ now gives the result.
 \end{proof}
Our next goal is to obtain  a formula for $M(k,g,n,r)$ by reducing to the case $g=1$. To that end, suppose that there are exactly $r$ success runs among the first $n$ trials.

Let $P_1$ denote the probability that the $r^{th}$ success run  contains trial $n$.  Note that  by Remark~\ref{rem: prog} and  Lemma~\ref{lem: reduction}\begin{equation} \label{eq: P1}
P_1 =  \frac{1}{q}X(k,g,n+1,r) = \frac{1}{q}  X(k,1,n+1-(r-1)(g-1),r). \end{equation}
Let $P_2$ denote the probability that there are exactly $r$ success runs and that the $r^{th}$ buffer has been completed by trial $n$.  To calculate $P_2$,  let $P_3$ denote the probability that  the $r^{th}$ buffer is completed exactly at trial $(n+1)$, so that the $r^{th}$ trial is completed exactly at trial $n+1-g$ which is followed by an $F$ trial.  (We can assume that we have an infinite sequence of independent trials so that it makes sense to consider trial $n+1$.)
Note that  by Lemma~\ref{lem: reduction}
\begin{equation} \label{eq: P3} P_3 = X(k,g,n+1-(g-1),r) = X(k,1,n+1-r(g-1),r) \end{equation} and hence
\begin{equation} \label{eq: M} M(k,1,n-r(g-1),r) = P_2 + \frac{1}{q}P_3. \end{equation}
To see this, note that the left-hand side of  eq.~\eqref{eq: M} is the probability that, for $g=1$,  a sequence of  $n - r(g-1)$ trials contains exactly $r$ success runs.
On the right-hand side, arguing as in Lemma~\ref{lem: reduction}, $P_2$ represents the probability that the $r^{th}$ 
success run has been completed,
while   by eq.~\eqref{eq: P3}  and Remark~\ref{rem: prog},
$P_3/q$ is the probability that the $r^{th}$ success run contains the final
($(n-r(g-1))^{th}$)
trial, i.e., the last success run  has not been completed. 

Let $P_4$ denote the probability that the $r^{th}$ buffer has begun but has not been completed by trial $n$. Note that
\begin{equation} \label{eq: P4} P_4 = \sum_{i=0}^{g-2} X(k,g,n-i,r) =  \sum_{i=0}^{g-2} X(k,1,n-(r-1)(g-1)-i,r). \end{equation}
The next theorem allows us to simplify the expression for $P_4$.

\begin{theorem} \label{thm: Y}
 Let $Y(k,n,r) = \sum_{m=1}^n  X(k,1,m,r)$.  Then
 \begin{equation*} Y(k,n,r) = \sum_{u=r-1}^{n-rk} p^{n-u}q^u \sum_{j = r}^u (-1)^j \binom{j-1}{r-1} \binom{u}{j}\binom{n-jk}{u}.
\end{equation*} \end{theorem}\begin{proof} Note that $Y(k,n,r)$ is the probability  of the event that  at least $r$ success runs (of length  at least $k$) have been completed  and followed by an $F$ trial up to and including trial $n$. Let $u$ denote the number of $F$ trials among the first $n$ trials. For this event to occur we obviously  require $r-1 \le u\le n - kr$. Note that these $u$ trials define exactly $u+1$ `slots' between them, including slots before the first and after the last $F$ trial, and that a completed success run can occupy  any slot except the $(u+1)^{th}$ slot.
 For $1 \le i \le u$, let  $E_i$ be the event that the $i^{th}$ slot contains a success run of length at least $k$ and let 
$$F_{u,j}=  \sum_{1 \le i_1 < i_2 < \dots < i_j\le u} P_u(E_{i_1}  \cap E_{i_2} \cap \dots \cap E_{i_j})\qquad(1 \le j \le u+1),$$
where $P_u$ denotes the conditional probability given exactly $u$ $F$ trials occur.
By a version of the inclusion exclusion principle (cf.~\cite{Feller}, p.~74),
the probability that there are at least $r$ completed success runs is given by 
$$\sum_{u} p^{n-u}q^u \sum_{j \ge r} (-1)^j \binom{j-1}{r-1}F_{u,j}.$$
  To compute $F_{u,j}$, note that a subcollection of size $j$  can be chosen from the first $u$ slots in $\binom{u}{j}$ ways.
  If each slot from this subcollection contains a success run of length  $k$ then there remain $n-jk-u$ $S$ trials which can be allocated arbitrarily to the $u+1$ slots.
  This can be done in $\binom{n-jk}{u}$ ways.
  Hence $$F_{u,j} = \binom{u}{j}\binom{n-jk}{u},$$
and the result follows.
\end{proof}

Theorem~\ref{thm: Y} and eq.~\eqref{eq: P4} yield the following  immediately:\begin{equation} \label{eq_ P4again}
 P_4 =  Y(k,n-(r-1)(g-1),r) - Y(k,n-(r-1)g,r). \end{equation}

\begin{theorem} \begin{equation*} \begin{split}
M(k,g,n,r) &=\frac{1}{q}[ X(k,1,n+1-(r-1)(g-1),r)-X(k,1,n+1-r(g-1),r)]\\ &+ M(k,1,n-r(g-1),r)\\
&+Y(k,n-(r-1)(g-1),r) - Y(k,n-(r-1)g,r). \end{split}
\end{equation*}
\end{theorem} \begin{proof} Note that
$$M(k,g,n,r) = P_1+P_2+P_4.$$ Hence the result follows by combining eqs.~\eqref{eq: P1}, \eqref{eq: P3}, \eqref{eq: M} and \eqref{eq: P4}.
\end{proof}
Theorem~\ref{thm: Y} also yields an expression for the distribution function of the number of success runs.
\begin{theorem}
\label{thm: Y_gt}
For given $k \ge 1$ and $g \ge 1$ the probability that there are at least $r$ success runs in the first $n$ trials is given by
\begin{equation}
\label{eq:SJD_P_N_gt_r}
  \sum_{j=r}^\infty M(k,g,n,j) =  Y(k,n-(r-1)(g-1),r)+  \frac{1}{q}X(k,1,n+1-(r-1)(g-1),r).
\end{equation}
\end{theorem}
\begin{proof} The probability that the $r^{th}$ success run has been completed by trial $n$ is given by
$$ \sum_{j=1}^{n} X(k,g,j,r),$$ while the probability that it is `in progress' at trial $n$, i.e.,  that trial $n$ is at least the $k^{th}$  $S$ trial of the $r^{th}$ success run,   is given by
$$\frac{1}{q}X(k,g,n+1,r).$$
Hence the required probability is given by \begin{align*}
& \sum_{j=1}^{n} X(k,g,j,r) + \frac{1}{q} X(k,g,n+1,r)\\ &= \sum_{j=1}^n X(k,1,j-(r-1)(g-1),r)+ \frac{1}{q}X(k,1,n+1-(r-1)(g-1),r)\\
&= Y(k,n-(r-1)(g-1),r)+  \frac{1}{q}X(k,1,n+1-(r-1)(g-1),r).
\end{align*}
\end{proof}

\begin{proposition}
Theorem \ref{thm: Y_gt}, i.e.~eq.~\eqref{eq:SJD_P_N_gt_r}, is equivalent to eq.~\eqref{eq:p_N_gt_r}.
This establishes the equivalence of the results in this section with those in Sec.~\ref{sec:pmf1}.
\end{proposition}
\begin{proof}
We first process the expression for $X(k,1,n,r)$ by interchanging the order of the sums.
\begin{equation}
\label{eq:Xk1nr_interchange}
\begin{split}
X(k,1,n,r) &= \sum_{u=r-1}^{n-1-kr}q^{u+1}p^{n-u-1} \biggl[\sum_{s=r-1}^u  (-1)^{s-r+1} \binom{u}{s} \binom {s}{r-1}\binom{n-1-(s+1)k}{u} \biggr] 
\\
&= qp^{n-1} \sum_{s} (-1)^{s-r+1}\binom{s}{r-1} \biggl[ \sum_{u=r-1}^{n-1-kr} \binom{u}{s} \binom{n-1-(s+1)k}{u} (q/p)^u \biggr] \,.
\end{split}
\end{equation}
Define $x = q/p$ and note that 
\begin{equation}
\sum_{u=0}^{n-1-(s+1)k} \binom{n-1-(s+1)k}{u} (\lambda x)^u = (1 + \lambda x)^{n-1-(s+1)k} \,.
\end{equation}
Differentiate $s$ times with respect to $\lambda$ and set $\lambda=1$ and recall $x=q/p$.
\begin{equation}
\begin{split}
\sum_{u=0}^{n-1-(s+1)k} \binom{u}{s} \binom{n-1-(s+1)k}{u} x^u &= x^s (1 + x)^{n-1-(s+1)k -s} \binom{n-1-(s+1)k}{s} 
\\
&= q^s p^{-n+1+(s+1)k} (p + q)^{n-1-(s+1)k -s} \binom{n-1-(s+1)k}{s}
\\
&= q^s p^{-n+1+(s+1)k} \binom{n-1-(s+1)k}{s} \,.
\end{split} 
\end{equation}
Return to eq.~\eqref{eq:Xk1nr_interchange} (set $m=s+1$ in the last line)
\begin{equation}
\label{eq:Xk1nr_simplified}
\begin{split}
X(k,1,n,r) &= qp^{n-1} \sum_{s} (-1)^{s-r+1}\binom{s}{r-1} q^s p^{-n+1+(s+1)k} \binom{n-1-(s+1)k}{s}
\\
&= qp^k \sum_{s=r-1}^{\lfloor(n-k-1)/(k+1)\rfloor} (-1)^{s-r+1}q^s p^{sk} \binom{s}{r-1} \binom{n-1-(s+1)k}{s} 
\\
&= \sum_{m=r}^{\lfloor n/(k+1)\rfloor} (-1)^{m-r}q^m p^{mk} \binom{m-1}{r-1} \binom{n-1-mk}{m-1} \,.
\end{split} 
\end{equation}
Next we process the expression for $Y(k,n,r)$ by interchanging the order of the sums.
\begin{equation}
\label{eq:Yknr_interchange}
\begin{split}
Y(k,n,r) &= \sum_u \sum_{t \ge r} (-1)^{t-r} \binom{t-1}{r-1} \binom{u}{t} \binom{n-tk}{u} p^{n-u}q^u 
\\
&= p^n \sum_{t \ge r} (-1)^{t-r}\binom{t-1}{r-1} \biggl[ \sum_{u=t}^{n-kt} \binom{u}{t} \binom{n-kt}{u} (q/p)^u \biggr] \,.
\end{split}
\end{equation}
Again define $x = q/p$ and note that 
\begin{equation}
\sum_{u=0}^{n-kt} \binom{n-kt}{u} (\lambda x)^u = (1 + \lambda x)^{n-kt} \,.
\end{equation}
Differentiate $t$ times with respect to $\lambda$ and set $\lambda=1$ and recall $x=q/p$.
\begin{equation}
\begin{split}
\sum_{u=0}^{n-kt} \binom{u}{t} \binom{n-kt}{u} x^u &= x^t (1 + x)^{n-kt -t} \binom{n-kt}{t} 
\\
&= q^t p^{-n+kt} (p + q)^{n-kt -t} \binom{n-kt}{t}
\\
&= q^t p^{-n+kt} \binom{n-kt}{t} \,.
\end{split} 
\end{equation}
Return to eq.~\eqref{eq:Yknr_interchange} (replace $t$ by $m$ in the last line).
\begin{equation}
\label{eq:Yknr_simplified}
\begin{split}
Y(k,n,r) &= p^n \sum_{t \ge r} (-1)^{t-r}\binom{t-1}{r-1} q^t p^{-n+kt} \binom{n-kt}{t} 
\\
&= \sum_{t=r}^{\lfloor n/(k+1)\rfloor} (-1)^{t-r}q^t p^{kt} \binom{t-1}{r-1} \binom{n-kt}{t} 
\\
&= \sum_{m=r}^{\lfloor n/(k+1)\rfloor} (-1)^{m-r}q^m p^{mk} \binom{m-1}{r-1} \binom{n-mk}{m} \,.
\end{split} 
\end{equation}
Then we obtain the following expression for $P(N_s \ge r)$.
Recall $\nu=n-(r-1)(g-1)$. Set the upper limits equal in both sums.
(This introduces extraneous terms in the sum for $Y(k,n-(r-1)(g-1),r)$, but the binomial coefficients equal zero because the numerator is less than the denominator.)
\begin{equation}
\begin{split}
P(N_s \ge r) &= \frac{1}{q} X(k,1,n+1-(r-1)(g-1),r) +Y(k,n-(r-1)(g-1),r)
\\
&= \frac1q \sum_{m=r}^{\lfloor (\nu+1)/(k+1)\rfloor} (-1)^{m-r}q^m p^{mk} \binom{m-1}{r-1} \binom{\nu-mk}{m-1} 
\\
&\qquad +\sum_{m=r}^{\lfloor \nu/(k+1)\rfloor} (-1)^{m-r}q^m p^{mk} \binom{m-1}{r-1} \binom{\nu-mk}{m} 
\\
&= \sum_{m=r}^{\lfloor (\nu+1)/(k+1)\rfloor} (-1)^{m-r}q^{m-1} p^{mk} \binom{m-1}{r-1} \biggl[ \binom{\nu-mk}{m-1} +q \binom{\nu-mk}{m} \biggr] 
\\
&= \sum_{m=r}^{\lfloor (\nu+1)/(k+1)\rfloor} (-1)^{m-r}q^{m-1} p^{mk} \binom{m-1}{r-1} \biggl[ \binom{\nu-mk+1}{m} +(q-1) \binom{\nu-mk}{m} \biggr]
\\
&= \sum_{m=r}^{\lfloor (\nu+1)/(k+1)\rfloor} (-1)^{m-r}q^{m-1} p^{mk} \binom{m-1}{r-1} \biggl[ \binom{\nu-mk+1}{m} -p \binom{\nu-mk}{m} \biggr] \,.
\end{split}
\end{equation}
This equals eq.~\eqref{eq:p_N_gt_r}.
\end{proof}

\setcounter{equation}{0}
\section{Negative binomial distribution}\label{sec:negbinom}
We briefly discuss the negative binomial distribution $\textrm{NB}_{\rm II}(k,g,r)$.
The number of success runs $r$ is fixed and the number of trials $n$ is variable.
By definition, the set of trials ends with a success run of length $k$.
It is preceded by $r-1$ success runs, hence there are $r-1$ gaps between the success runs, each of which contains one failure followed by $g-1$ arbitrary outcomes.
Hence there are totally $(r-1)(g-1)$ arbitrary outcomes.
Observe that $\textrm{NB}_{\rm II}(k,g,r)$ has the same probability distribution as $\textrm{NB}_{\rm II}(k,1,r) + (r-1)(g-1)$.
Omitting unnecessary arguments, the mean and variance are therefore given by 
\begin{subequations}
\begin{align}
\mu(g) &= \mu(1) + (r-1)(g-1) \,.
\\
\sigma^2(g) &= \sigma^2(1) \,.
\end{align}
\end{subequations}
For $g=1$, Balakrishnan and Koutras (\cite{BalakrishnanKoutras}, eq.~(4.19)) said the mean and variance are
\begin{subequations}
\begin{align}
\label{eq:negbinom_mean_BK}
\mu &= \frac{r-p^k}{qp^k} \,,
\\
\label{eq:negbinom_var_BK}
\sigma^2 &= \frac{r[1-(2k+1)qp^k -p^{2k+1}]}{(qp^k)^2} +\frac{(r-1)p}{q^2} \,.
\end{align}
\end{subequations}
\begin{remark}
All the higher moments (centered on the mean) are the same as for $g=1$.
\end{remark}
The probability mass function for $g=1$ was derived by Muselli (\cite{Muselli}, Theorem 4).
For $g>1$, we make the transcription $n \gets \nu$, where recall $\nu=n-(r-1)(g-1)$.
Then the probability mass function $\tilde{p}_{n,r}$ (omitting mention of $k$ and $g$) is as follows.
(We affix the tilde to avoid confusion with the pmf in Sec.~\ref{sec:pmf1}.)
\begin{equation}
\label{eq:Muselli_thm4}
\tilde{p}_{n,r} = \sum_{m=r}^{\left\lfloor\frac{\nu+1}{k+1}\right\rfloor} (-1)^{m-r}\binom{m-1}{r-1} p^{mk}q^{m-1}
\biggl[\binom{\nu-mk}{m-1} -p\binom{\nu-mk-1}{m-1}\biggr] \,.
\end{equation}
We take the opportunity here to present expressions for all the factorial moments for $g=1$.
Balakrishnan and Koutras presented an expression for the pgf.
First they defined (\cite{BalakrishnanKoutras}, eq.~(4.16)) $A(z) = (1-p^kz^k)/(1-pz)$.
Then the pgf is (\cite{BalakrishnanKoutras}, unnumbered after eq.~(4.18)) 
\begin{equation}
\label{eq:BK_pgf}
\begin{split}
G(z) &= \frac{p^{rk}q^{r-1}z^{rk+r-1}}{(1-qzA(z))^r(1-pz)^{r-1}}
\\
&= \frac{p^{rk}q^{r-1}z^{rk+r-1}(1-pz)}{(1-pz-qz(1-p^kz^k))^r}
\\
&= \frac{p^{rk}q^{r-1}z^{rk+r-1}(1-pz)}{(1-z+qp^kz^{k+1})^r} \,.
\end{split}
\end{equation}
The $s^{th}$ factorial moment is given by $\tilde{F}_{(s)} = [d^sG(z)/dz^s]_{z=1}$.
(We affix the tilde to avoid confusion with the factorial moments in Sec.~\ref{sec:facmom}.)
Omitting the details of the algebra, the result is
\begin{equation}
\label{eq:NB_facmom}
\tilde{F}_{(s)} = \frac{s!}{q^{s+1}p^{ks}}\sum_{j=0}^s (-1)^jq^jp^{jk}\binom{s-j+r-1}{r-1}\biggl[\binom{(k+1)s-jk}{j} -p\binom{(k+1)s-jk-1}{j} \biggr] \,.
\end{equation}
Note the following.
\begin{enumerate}
\item
Observe that because of the numerator factor $1-pz$ in the pgf in eq.~\eqref{eq:BK_pgf}, 
the factorial moment $\tilde{F}_{(s)}$ also has the structure ``$f(n)-pf(n-1)$'' where in this case ``$n$'' is $(k+1)s$.
\item
The case $r=1$ is the geometric distribution of order $k$.
For this case, Balakrishnan and Koutras (\cite{BalakrishnanKoutras}, unnumbered before eq.~(2.18)) gave a recurrence for the factorial moments but not an explicit solution.
\item
The factorial moments $\tilde{F}_{(s)}(g)$ for $g>1$ are given by ($c=(r-1)(g-1)$ and $c_{(j)}$ is the falling factorial).
\begin{equation}
\tilde{F}_{(s)}(g) = \sum_{j=0}^{\min(s,(r-1)(g-1))} \binom{s}{j} c_{(j)}\,\tilde{F}_{(s-j)}(1) \,.
\end{equation}
\end{enumerate}

\setcounter{equation}{0}
\section{Distribution of longest success run II}\label{sec:longestrun2}
\subsection{General remarks}
We present additional results for the distribution of the longest success run.
Recall that in Sec.~\ref{sec:longestrun1} we computed $P(L_n \le t)$, where $L$ is the length of the longest run, for fixed $(k,g,n)$.
Here we focus attention on $P(L_n \ge t)$.
We derive a recurrence for $P(L_n \ge t)$ and the associated generating function.
Using the generating function, we derive a combinatorial sum for $P(L_n \ge t)$ for $g\ge1$.
We show that the generating function and combinatorial sum simplify if $g=1$.
We also derive the mean $\mathbb{E}[L_n]$ and variance $\textrm{Var}[L_n]$ and factorial moments $\mathbb{E}[L_n(L_n-1)]$, etc.,  for $g\ge1$.

\subsection{Recurrence \&\ generating function for longest run}
Kopocinsky treated the case $g=1$ and derived a recurrence for $P_n(L_n \ge t)$
(\cite{Kopocinsky} Theorem 1) and the associated generating function
(\cite{Kopocinsky} Theorem 3).
We treat $g\ge1$ below.
Fix $k\ge1$ and $g\ge1$. Let $P_n(L_n \ge t)$ be the probability that the longest $k$-run with gap $g$ is at least $t$, where $t\ge k$.
We present a recurrence for $f(n,t) = P_n(L_n \ge t)$ for fixed $(k,g,t)$.
\begin{proposition}\label{prop_SJD_rec}
For $n>t$, 
\begin{equation}
\label{eq:rec_fnt}
f(n,t) = q\sum_{i=0}^{k-1} p^if(n-i-1,t) +q\sum_{i=k}^{t-1} p^if(n-i-g,t) +p^t \,,
\end{equation}
with the initial conditions $f(t,t)=p^t$ and $f(n,t)=0$ for $-\infty < n < t$.
\end{proposition}
\begin{proof}
  We condition on the first $t$ trials.
\vskip 0.1in\noindent  
\textbf{Case I:} Suppose that $0 \le i < k$, the first $i$ trials are $S$ and the $(i+1)^{th}$ trial is $F$.
Since a success has length at least $k$ it follows that the conditional probability that $L_n \ge t$ is just $f(n-i-1,t)$.

\vskip 0.1in\noindent  
\textbf{Case II:} Suppose that $k \le i \le t-1$, the first $i$ trials are $S$ and the $(i+1)^{th}$ trial is $F$.
In this case a success run has occurred and therefore a gap of length $g$ is required after this success run.
Hence the conditional probability that $L_n \ge t$ is $f(n-i-g,t)$.

\vskip 0.1in\noindent  
\textbf{Case III:} Suppose the first $t$ trials are $S$. Then, clearly, the conditional probability that $L_n \ge t$ is $1$.

\vskip 0.1in\noindent  
Hence the result follows.
\end{proof}
Next we derive the generating function $F(z,t) = \sum_{n=0}^\infty f(n,t)z^n$.
First we recall standard formalism for the theory of an inhomogeneous linear recurrence with constant coefficients.
Suppose that
\begin{equation}
\label{eq:inhom_rec_const_coeff}
f(n) = a_1 f(n-1) +\dots +a_tf(n-t) + c \,.
\end{equation}
All the coefficients $a_1,\dots,a_t$ and $c$ are constants.
Define a generating function 
$G(z) = \sum_{n=0}^\infty z^n f(n)$.
We suppose $f(n)=0$ for $-\infty < n<t$.
Then from eq.~\eqref{eq:inhom_rec_const_coeff} we see that $f(t) = c$.
Standard manipulations yield
\begin{equation}
\begin{split}
G(z) &= (a_1z +\dots +a_tz^t) G(z) + \frac{cz^t}{1-z} 
\\
\Rightarrow \quad 
G(z) &= \frac{1}{1-z} \frac{cz^t}{1-a_1z-\dots a_tz^t} \,.
\end{split}
\end{equation}
For our particular application, the generating function is $F(z,t)$ and the coefficients are given in eq.~\eqref{eq:rec_fnt}.
Then
\begin{equation}
\label{eq:genfcn_Fzt}
\begin{split}
F(z,t) &= \frac{1}{1-z} \frac{p^tz^t}{1 -q(z +pz^2 +\dots +p^{k-1}z^k) -q(p^kz^{k+g} +\dots +p^{t-1}z^{t-1+g})}
\\
&= \frac{1}{1-z} \frac{p^tz^t}{1 -qz(1-p^kz^k)/(1-pz) -qp^kz^{k+g}(1-p^{t-k}z^{t-k})/(1-pz)}
\\
&= \frac{1}{1-z} \frac{(1-pz)p^tz^t}{1-pz -qz(1-p^kz^k) -qp^kz^{k+g}(1-p^{t-k}z^{t-k})}
\\
&= \frac{1-pz}{1-z} \frac{p^tz^t}{1-z +qp^k(z^{k+1} -z^{k+g}) +qp^tz^{t+g}} \,.
\end{split}
\end{equation}
\begin{remark}
For $g=1$, $F(z,t)$ simplifies
\begin{equation}
\label{eq:genfcn_g1}
F(z,t,g=1) = \frac{1-pz}{1-z}\,\frac{p^tz^t}{1-z +qp^tz^{t+1}} \,.
\end{equation}  
This equals Kopocinsky's expression (\cite{Kopocinsky} Theorem 3), with suitable changes of notation.
\end{remark}

\subsection{Combinatorial sum for distribution of longest run}
The factor $(1-pz)$ in eq.~\eqref{eq:genfcn_Fzt} means that $F(z,t) = F_1(z,t) -pzF_1(z,t)$.
Hence it suffices to expand $F_1(z,t)$ in powers of $z$
\begin{equation}
\begin{split}
F_1(z,t) &= \frac{1}{1-z}\,\frac{p^tz^t}{1-z +qp^k(z^{k+1}-z^{k+g}) +qp^tz^{t+g}}
\\
&= \frac{p^tz^t}{(1-z)^2}\,\frac{1}{1 +q(p^k(z^{k+1}-z^{k+g}) +p^tz^{t+g})/(1-z)}
\\
&= p^tz^t\sum_{r=0}^\infty (-1)^r q^r \frac{(p^k(z^{k+1}-z^{k+g}) +p^tz^{t+g})^r}{(1-z)^{r+2}}
\\
&= p^tz^t\sum_{r=0}^\infty (-1)^r q^r \biggl(\sum_{j_1+j_2+j_3=r} (-1)^{j_2} \binom{r}{j_1,j_2,j_3} p^{(j_1+j_2)k+j_3t} z^{j_1(k+1) +j_2(k+g) +j_3(t+g)} \biggr)\;\times
\\
&\qquad\qquad\qquad \biggl(\sum_{s=0}^\infty \binom{s+r+1}{r+1} z^s\biggr)
\\
&= p^tz^t\sum_{r=0}^\infty (-1)^r q^r \biggl(\sum_{j_1+j_2+j_3=r} (-1)^{j_2} \binom{r}{j_1,j_2,j_3} p^{rk+j_3(t-k)} z^{j_1(k+1) +j_2(k+g) +j_3(t+g)} \biggr)\;\times
\\
&\qquad\qquad\qquad \biggl(\sum_{s=0}^\infty \binom{s+r+1}{r+1} z^s\biggr)
\end{split}
\end{equation}
The coefficient of $z^n$ is given by tuples $(j_1,j_2,j_3,s)$ such that 
$n = s +t +j_1(k+1)+j_2(k+g)+j_3(t+g)$, hence
$s = n -t -j_1(k+1)-j_2(k+g)-j_3(t+g)$.
Denote the coefficient of $z^n$ by $\zeta(n)$. Then
\begin{equation}
\begin{split}
  \zeta(n) &= \sum_{r=0}^\infty (-1)^r q^rp^{t+kr} \;\times
  \\
  &\qquad \sum_{j_1+j_2+j_3=r} (-1)^{j_2} p^{j_3(t-k)} \binom{r}{j_1,j_2,j_3} \binom{n -t -j_1(k+1)-j_2(k+g)-j_3(t+g) +r+1}{r+1} \,.
\end{split}
\end{equation}
The sum over $r$ terminates when $n -t -j_1(k+1)-j_2(k+g)-j_3(t+g) < 0$.
Then
$f(n,t) = \zeta(n) - p\zeta(n-1)$.

We can write the sums more concisely using only $j_1$, $j_2$ and $j_3$ as follows.
\emph{Define} $r=j_1+j_2+j_3$ and recall $s = n -t -j_1(k+1)-j_2(k+g)-j_3(t+g)$. Then
\begin{equation}
\label{eq:fnt_combin}
\begin{split}
  f(n,t) &= \sum_{j_1=0,\,j_2=0,\,j_3=0}^{s \ge 0}
  (-1)^{j_1+j_3} q^r p^{(j_1+j_2)k+(j_3+1)t} \binom{r}{j_1,j_2,j_3} \biggl[\binom{s +r+1}{r+1} -p\binom{s +r}{r+1} \biggr] \,.
\end{split}
\end{equation}

The sum simplifies for $g=1$.
We employ eq.~\eqref{eq:genfcn_g1}.
Again this has the form $F(z,t) = F_1 - pzF_1(z)$, hence we process only $F_1$. Then
\begin{equation}
\begin{split}
F_1(z,t) &= \frac{1}{1-z}\,\frac{p^tz^t}{1-z +qp^tz^{t+1}}
\\
&= \frac{p^tz^t}{(1-z)^2}\,\frac{1}{1 +qp^tz^{t+1})/(1-z)}
\\
&= p^tz^t\sum_{r=0}^\infty (-1)^r q^r p^{rt} \frac{z^{r(t+1)}}{(1-z)^{r+2}}
\\
&= p^tz^t\sum_{r=0}^\infty (-1)^r q^rp^{rt}z^{r(t+1)} \sum_{s=0}^\infty \binom{s+r+1}{r+1} z^s \,.
\end{split}
\end{equation}
The coefficient of $z^n$ is given by tuples $(r,s)$ such that $n = s +t +r(t+1)$.
Hence $s = n -t -r(t+1)$. Because $s\ge0$, we must have $r \le \lfloor(n-t)/(t+1)\rfloor$.
Denote the coefficient of $z^n$ by $\zeta_1(n)$ (for ``$g=1$''). Then
\begin{equation}
\begin{split}
  \zeta_1(n) &= \sum_{r=0}^{\lfloor(n-t)/(t+1)\rfloor} (-1)^r q^rp^{(r+1)t} \binom{n -t -r(t+1)+r+1}{r+1} 
  \\
  &= \sum_{r=0}^{\lfloor(n-t)/(t+1)\rfloor} (-1)^r q^rp^{(r+1)t} \binom{n -(r+1)t +1}{r+1} \,.
\end{split}
\end{equation}
Then $f(n,t) = \zeta_1(n) - p\zeta_1(n-1)$, whence
\begin{equation}
f(n,t,g=1) = \sum_{r=0}^{\lfloor(n-t)/(t+1)\rfloor} (-1)^r q^rp^{(r+1)t} \biggl[\binom{n -(r+1)t +1}{r+1} -p\binom{n -(r+1)t}{r+1}\biggr] \,.
\end{equation}  
Set $m=r+1$ and the final expression is
\begin{equation}
\label{eq:fnt_combin_g1}
f(n,t,g=1) = \sum_{m=1}^{\lfloor(n+1)/(t+1)\rfloor} (-1)^{m-1} q^{m-1}p^{mt} \biggl[\binom{n -mt +1}{m} -p\binom{n -mt}{m}\biggr] \,.
\end{equation}  
This is identical to de Moivre's result (see eq.~\eqref{eq:P_L_ge_k}, with minor changes of notation).

\subsection{Mean, variance and factorial moments of longest run}\label{sec:mean_var_facmom_longestrun}
Even for $g=1$, there are no published formulas for the mean and variance of the longest success run, i.e.~$\mathbb{E}[L_n]$ and $\textrm{Var}[L_n]$ for fixed $(k,g,n)$.
However, Schilling published an approximate estimate for $\mathbb{E}[L_n]$ for $k=g=1$ (\cite{Schilling} eq.~(1))
\begin{equation}
\label{eq:Schilling_mean_est}
\mathbb{E}[L_n] \simeq \frac{\ln(nq)}{\ln(1/p)} \,.
\end{equation}  
See also additional results by Gordon, Schilling and Waterman \cite{GordonSchillingWaterman}.

\begin{proposition}
For fixed $k\ge1$, $g\ge1$ and $p\in(0,1)$, 
\begin{equation}
\mathbb{E}[L_n] \asymp \frac{\ln{n}}{\ln(1/p)} \,.
\end{equation}
\end{proposition}
\begin{proof}
  Note that if $m\ge k$ then each group of $m+g$ consecutive $S$ trials contains a portion of a success run of length at least $m$.
  Hence, for $\lambda>1$, if $n \ge \lambda(m+g)p^{-(m+g)}$ then 
\begin{equation}
\begin{split}
  P(L_n \ge m) &\ge 1 - (1-p^{m+g})^{\lambda p^{-(m+g)}}
  \\
  &\ge 1 - e^{-\lambda} 
\end{split}
\end{equation}
and hence $\mathbb{E}[L_n] \ge m(1-e^{-\lambda})$.
Note that
$\ln(n) \ge \ln(\lambda(m+g)) +(m+g)\ln(1/p)$ and so we may assume
\begin{equation}
m \asymp \frac{\ln{n}}{\ln(1/p)} \,.
\end{equation}
Letting $\lambda\to\infty$, 
\begin{equation}
\liminf_{n\to\infty} \frac{\mathbb{E}[L_n]}{\ln(n)} \ge  \frac{1}{\ln(1/p)} \,.
\end{equation}
On the other hand, clearly $L_n$ for fixed $k$ and $g$ is dominated by $L_n$ for $k=g=1$. So
\begin{equation}
\limsup_{n\to\infty} \frac{\mathbb{E}[L_n]}{\ln(n)} \le  \frac{1}{\ln(1/p)} 
\end{equation}
follows from the known case of $k=g=1$.
Combining the upper and lower estimates gives
\begin{equation}
\mathbb{E}[L_n] \asymp \frac{\ln{n}}{\ln(1/p)} \,.
\end{equation}
This proves the result.
\end{proof}

We derive expressions for the mean, variance and factorial moments as follows, for fixed $(k,g,n)$.
Note that $P(L_n = t) = P(L_n \ge t) - P(L_n \ge t+1)$.
For the convenience of the reader, we note the following well-known expressions for the factorial moments
$F_{(r)} = \mathbb{E}[L_n(L_n-1)\dots(L_n-r+1)]$.
\begin{subequations}
\begin{align}
  F_{(0)} &= 1 \,, \\
  F_{(1)} &= \sum_{t=1}^n P(L_n \ge t) \,, \\
  F_{(2)} &= 2\sum_{t=2}^n (t-1)P(L_n \ge t) \,, \\
  F_{(3)} &= 3\sum_{t=3}^n (t-1)(t-2)P(L_n \ge t) \,, \\
  &\;\vdots \\
  F_{(r)} &= r\sum_{t=r}^n (t-1)(t-2)\dots(t-r+1)P(L_n \ge t) \nonumber\\
  &= r!\sum_{t=r}^n \binom{t-1}{r-1}P(L_n \ge t) \,.
\end{align}
\end{subequations}
There are exactly $n+1$ nonzero factorial moments, indexed by $r=0,1,\dots,n$.
If $r>n$ then $F_{(r)}=0$.
Hence the mean and variance are given by
\begin{subequations}
\label{eq:mean_var_longestrun}
\begin{align}
  \mathbb{E}[L_n] &= \sum_{t=1}^n P(L_n \ge t) \,,
  \\
  \textrm{Var}[L_n] &= \mathbb{E}[L_n(L_n-1)] +\mathbb{E}[L_n] -(\mathbb{E}[L_n])^2
  \nonumber \\
  &= \biggl(\sum_{t=1}^n (2t-1) P(L_n \ge t)\biggr) -\biggl(\sum_{t=1}^n P(L_n \ge t)\biggr)^2 \,.
\end{align}
\end{subequations}
Observe that $P(L_n = t)$ is a probability mass and $P(L_n \le t)$ is a cumulative probability mass and $P(L_n \ge t)$ is the complement cumulative probability mass.
The above expressions for $\mathbb{E}[L_n]$ and $\textrm{Var}[L_n]$ in terms of weighted sums over $P(L_n \ge t)$
are general formulas for any integer-valued discrete distribution with support $0,1,2,\dots$.
Our contribution is to provide an explicit expression for $P(L_n \ge t)$ for the case of the longest success run in Bernoulli trials.

\begin{remark}
For $g=1$, eq.~\eqref{eq:fnt_combin_g1} yields the following expression for the mean.
\begin{equation}
\mathbb{E}[L_n] = \sum_{t=1}^n \sum_{m=1}^{\lfloor(n+1)/(t+1)\rfloor} (-1)^{m-1} q^{m-1}p^{mt} \biggl[\binom{n -mt +1}{m} -p\binom{n -mt}{m}\biggr] \,.
\end{equation}  
\end{remark}
  
\begin{remark}
Recall that if $k \le n \le 2k$, only one success run of length at least $k$ is possible.
In that case, eq.~\ref{eq:P_L_ge_t_onerun} furnishes a simple expression for $P(L_n \ge t)$.
It is then possible to derive closed-form expressions for the mean and variance.
The details of the algebra are tedious but straightforward and are omitted.
The mean is given by
\begin{equation}
  \mathbb{E}[L_n] = p^k [n +(n-k)(k-1)q] \,.
\end{equation}
From eq.~\eqref{eq:mean_var_longestrun}, for the variance we require the sum 
\begin{equation}
\begin{split}
  S &= \sum_{t=1}^n (2t-1)P(L_n \ge t) 
  \\
  &= kp^k[ k +2(n-k)p +k(n-k)q ] +\frac{p}{q^2} \Bigl[(n-k)q p^k (1+p) -2 p (p^k - p^n) \Bigr] \,.
\end{split}
\end{equation}
The variance is obtained by subtraction
\begin{equation}
\begin{split}
  \textrm{Var}[L_n] &= \biggl(\sum_{t=1}^n (2t-1)P(L_n \ge t) \biggr) - (\mathbb{E}[L])^2
  \\
  &= S - p^{2k} [n +(n-k)(k-1)q]^2 \,.
\end{split}
\end{equation}
\end{remark}

\setcounter{equation}{0}
\section{Distribution of longest success run III}\label{sec:longestrun3}
In this section, we restrict attention to $g=1$.
We make a crucial observation.
Dilworth and Mane \cite{DM3} treated the geometric distribution of order $k$,
and employed the notation $f_k(n)$ to denote the probability mass function for the first run of $k$ consecutive successes, ending at trial $n$.
Then observe that for fixed $n$,
\begin{equation}
\label{eq:P_L_DM3sum}
P(L_n \le k-1) = \sum_{i=n+1}^\infty f_k(i) \,.
\end{equation}
The recurrence for $f_k(n)$ is (\cite{DM3}, eq.~(1))
\begin{equation}
\label{eq:rec_fnk}
f_k(n) = q\sum_{i=0}^{k-1} p^if_k(n-i-1) \,.
\end{equation}
The auxiliary polynomial of the recurrence in eq.~\eqref{eq:rec_fnk} is
\begin{equation}
\label{eq:aux_a}
\mathscr{A}(z) = z^k - q(z^{k-1} +pz^{k-2} +\dots +p^{k-1}) \,.
\end{equation}
Following \cite{DM3}, we denote the roots of $\mathscr{A}(z)$ by $\lambda_j$, where $j=0,\dots,k-1$.
All the roots are simple.
The root $\lambda_0$ is the unique real positive root and is called the `principal root' (and the others are termed `secondary roots').
Additional properties of the roots are listed in \cite{DM3}.
Dilworth and Mane (\cite{DM3}, Theorem 1) presented expressions for the roots as Fuss-Catalan series.
For $n\ge1$, let
\begin{equation}
b_n = \frac{\Gamma(n-1+n/k)}{n!\Gamma(n/k)}\,.
\end{equation}
For $p\in(0,1)$, the secondary roots are given by
\begin{equation}
\lambda_j = \sum_{n=1}^\infty b_n \, (e^{2\pi ij/k}pq^{1/k})^n \qquad (1 \le j < k) \,.
\end{equation}
The principal root $\lambda_0$ is given by
\begin{equation}
\lambda_0 = \begin{cases} \displaystyle \phantom{\biggl|}1 - k\sum_{m=1}^\infty b_{mk} \, (p^kq)^m  &\qquad p < k/(k+1) \,,
\\
\phantom{\biggl|} k/(k+1) &\qquad  p = k/(k+1) \,,
\\
\displaystyle \phantom{\biggl|} \sum_{n=1}^\infty b_n \, (pq^{1/k})^n &\qquad p > k/(k+1) \,.
\end{cases}
\end{equation}
Dilworth and Mane derived a Binet-style sum for $f_k(n)$ (\cite{DM3}, eq.~(58))
\begin{equation}
\label{eq:fncj}
f_k(n) = \sum_{j=0}^{k-1} c_j\, \lambda_j^{n-1}  \,.
\end{equation}
Here the $c_j$ are a set of constant coefficients, i.e.,~they depend on $p$ and $k$ but not on $n$.
Then from eq.~\eqref{eq:P_L_DM3sum},
\begin{equation}
\begin{split}
\label{eq:P_L_DM3sum1}
P(L_n \le k-1) &= \sum_{i=n+1}^\infty \sum_{j=0}^{k-1} c_j\, \lambda_j^{i-1}
\\
&= \sum_{j=0}^{k-1} c_j\, \frac{\lambda_j^n}{1-\lambda_j} \,.
\end{split}
\end{equation}
Expressions for the $c_j$ were derived in \cite{DM3}.
For $p\ne k/(k+1)$, the result is
\begin{equation}
\label{eq:P_Binet_fixed_n}  
P(L_n \le k-1) = \frac{p^k}{k+1} \sum_{j=0}^{k-1} \frac{\lambda_j^{n-k+1}}{1-\lambda_j} \frac{\lambda_j-p}{\lambda_j -k/(k+1)} \,.
\end{equation}
If $p=k/(k+1)$, the result is
\begin{equation}
\label{eq:P_Binet_fixed_n_pcrit}  
P(L_n \le k-1) = \frac{p^k}{k+1} \biggl( \frac{2\lambda_0^{n-k+1}}{1-\lambda_0} + \sum_{j=1}^{k-1} \frac{\lambda_j^{n-k+1}}{1-\lambda_j} \biggr) \,.
\end{equation}
This yields an expression for $P(L_n \le k-1)$ as a Binet-style sum over powers of the roots of the auxiliary polynomial.

\setcounter{equation}{0}
\section{Conclusion}\label{sec:conc}
This paper is an extension of the exposition by Dilworth and Mane in \cite{DM5}.
We presented two derivations for the probability mass function to attain $r$ success runs, for fixed $(k,g,n)$,
in Secs.~\ref{sec:pmf1} and \ref{sec:pmf2}, respectively.
Using the insights in this note, we revisited the results for the factorial moments derived in \cite{DM5}
and rewrote them in a more concise and illuminating manner.
Let $L$ denote the length of the longest success run.
We derived the distribution of the longest success run $P(L \le t)$.
We presented results for the mean, variance, probability mass function and factorial moments for
$\textrm{NB}_{\rm II}(k,g,r)$, the Type II negative binomial distribution of order $k$,
where the number of success runs $r$ is fixed and the number of trials $n$ is variable.
Next, we presented a recurrence and generating function for $P(L \ge t)$, 
and expressions for the mean, variance and factorial moments of $L$.
In Sec.~\ref{sec:longestrun3}, we revisted the analysis of the longest success run $P(L \le k-1)$, for the special case $g=1$.
We solved the auxiliary equation for the recurrence for $P(L \le k-1)$, using the formalism developed in \cite{DM3},
and derived a Binet-style sum for $P(L \le k-1)$.

\section*{Funding}
S.~J.~Dilworth was supported by Simons Foundation Collaboration Grant 849142.


\end{document}